\documentclass[12pt,reqno] {amsart}%
\usepackage{amsfonts}
\usepackage{amsmath}
\usepackage{amssymb}
\usepackage{xypic,color}
\usepackage{graphicx}%
\usepackage{hyperref}
\usepackage{cleveref}
\usepackage{comment}
 \usepackage{mathrsfs}
 \usepackage{float}
 \usepackage{tikz}
 \usepackage{bm}
 \usepackage{ulem}
\theoremstyle{plain}

\newtheorem{question}{Question}

\newtheorem{remark}{Remark}[section]
\newtheorem{example}{Example}[section]

\newtheorem{conjecture}{Conjecture}[section]
\newtheorem{definition}{Definition}[section]
\numberwithin{equation}{section}
\newtheorem{theorem}{Theorem}[section]
\newtheorem{corollary}[theorem]{Corollary}
\newtheorem{lemma}[theorem]{Lemma}
\newtheorem{proposition}[theorem]{Proposition}
\newtheorem*{ac}{Acknowledgement}

\newtheorem{main theorem}{Main Theorem}

\begin{document}

\title{Quantum convolution inequalities on Frobenius von Neumann algebras}

\author{Linzhe Huang}
\address{L. Huang, Yau Mathematical Sciences Center, Tsinghua University, Beijing, 100084, China}
\email{huanglinzhe@mail.tsinghua.edu.cn}

\author{Zhengwei Liu}
\address{Z. Liu, Yau Mathematical Sciences Center and Department of Mathematics, Tsinghua University, Beijing, 100084, China}
\address{Beijing Institute of Mathematical Sciences and Applications, Huairou District, Beijing, 101408, China}
\email{liuzhengwei@mail.tsinghua.edu.cn}

\author{Jinsong Wu}
\address{J. Wu, Beijing Institute of Mathematical Sciences and Applications, Beijing, 101408, China}
\email{wjs@bimsa.cn}
\date{}

\maketitle

\begin{abstract}
In this paper, we introduce Frobenius von Neumann algebras and study quantum convolution inequalities. In this framework, we unify quantum Young's inequality on quantum symmetries such as subfactors, and fusion bi-algebras studied in quantum Fourier analysis. 
Moreover, we prove quantum  entropic convolution inequalities and characterize the extremizers in the subfactor case. We also prove quantum smooth entropic convolution inequalities.
We obtain the positivity of comultiplications of subfactor planar algebras, which is stronger than the quantum Schur product theorem. All these inequalities provide analytic obstructions of unitary categorification of fusion rings stronger than Schur product criterion. 
\end{abstract}


\bigskip
\section{Introduction}
In \cite{Bec75}, Beckner remarkably proved the sharp Young's inequality \cite{You12}.
Recently, quantum Young's inequality for convolution has been established on quantum symmetries, such as
subfactors \cite{JLW16,BVG21}, fushion bi-algebras \cite{LPW}, Kac algebras \cite{LW18}, locally compact quantum groups 
\cite{JLW18},  etc., see further discussions 
 in the framework of quantum Fourier analysis \cite{JJLRW}.
Moreover, the extremal pairs of quantum Young’s
inequality  are characterized on subfactor planar algebras \cite{JLW19} and kac algebras \cite{LW18}.

In 2021, A. Wigderson and Y. Wigderson \cite{WW21} unified various classical (smooth) uncertainty principles related to discrete Fourier transforms. In \cite{HLW21}, the authors unified various quantum (smooth) uncertainty principles on quantum symmetries related to quantum Fourier transforms. 

One  goal of this paper is to unify inequalities  on quantum symmetries related to the convolution. 
The study of the comultiplication turns out to be more   fundamental than the study of convolution.
\begin{theorem}[See Theorem \ref{thm:positivity of comultiplication on subfactor}]\label{thm:main thm 1}
Suppose $\mathscr{P}_{\bullet}$ is a subfactor planar algebra. Then the comultiplication $\Delta$: $\mathscr{P}_{2,\pm}\rightarrow\mathscr{P}_{2,\pm}\otimes\mathscr{P}_{2,\pm}$
is positive.
\end{theorem}
Theorem \ref{thm:main thm 1} implies  the positivity of convolution (called quantum Schur product theorem, see Theorem 4.1 in \cite{Liu}), which led to an analytic criterion of unitary categorification of fusion rings (called Schur product criterion, see Proposition 8.3 in \cite{LPW}).
Theorem \ref{thm:main thm 1} leads to a stronger criterion of  unitary categorification.
\begin{theorem}[Positivity of Comultiplication Criterion]
For a fusion ring $\mathfrak{A}$, let $N_{k,j}^i$ be the fusion coefficients and $M_k=(N_{k,j}^i)_{i,j}$ be the fusion matrices, $1\leq i,j,k\leq n$. 
If $\mathfrak{A}$ admits a unitary categorification,
then
 for any $v\in\mathbb{C}^n$,
\begin{align}\label{eq:positive comultiplication criterion introduction}
\sum_{k=1}^n\frac{v^* M_k v}{\|M_k\|}
M_k\otimes M_k\geq 0.
\end{align}
\end{theorem}

To study convolution inequalities on quantum symmetries captured by a tracial von Neumann algebra $(\mathcal{M},\tau)$, we introduce a comultiplication $\Delta$: $\mathcal{M}$ $\to$  $\mathcal{M}\overline{\otimes}\mathcal{M}$ and a
 convolution $\ast$:  $L^1(\mathcal{M})\times L^1(\mathcal{M})$ $\to$ $L^1(\mathcal{M})$, which are dual to each other. We call $\Delta$ a \textbf{good $k$-comultiplication} if $\Delta\geq0$ and $\Delta(I)=kI$, $k>0$. It induces a \textbf{good $k$-convolution} (See Proposition \ref{prop:comultiplication positivity imply convolution positivity}), namely
\begin{enumerate}
\item Positivity:
\[ x\ast y\geq0,\quad \forall\ x,y\geq0,\ x,y\in\mathcal{M};\]
\item Primary Young's inequality:
\[\|x\ast y\|_{1}\leq k\|x\|_1\|y\|_1,\quad \forall\ x,y\in\mathcal{M};\]
\item Haar measure:
\[\tau(x\ast y)=k\tau(x)\tau(y),\quad \forall\ x,y\geq0,\ x,y\in\mathcal{M}.\]
\end{enumerate}
 Furthermore,  
a trace-preserving anti $*$-isomorphism $\rho$ on $\mathcal{M}$ is called an $\textbf{antipode}$ if 
 the Frobenius Reciprocity holds:
\[\tau((x\ast y)z)=\tau 
((\rho(z)\ast x)\rho(y)),\quad \forall x,y,z\in\mathcal{M}.\]
We call the quadruple $(\mathcal{M},\tau,\ast,\rho)$ a \textbf{Frobenius von Neumann $k$-algebra (FN $k$-algebra)}. This notion is highly inspired by subfactor theory and Hopf algebras. Specific examples of FN $k$-algebras come from subfactor planar algebras, fusion bi-algebras, Kac algebras, etc. 
In Section \ref{sec:Implement of Young inequality}, we prove the  quantum Young's inequality on FN 
$k$-algebras, which unifies the quantum Young's inequality on subfactor planar algebras (See Theorem 4.13 in \cite{JLW16}) and fusion bi-algebras (See Theorem 5.11 in \cite{LPW}).
\begin{theorem}[See Theorem \ref{thm:quantum Young inequality}]\label{thm:main thm 2}
Let $(\mathcal{M},\tau,\ast,\rho)$ be a  FN $k$-algebra. Then for any $x,y\in\mathcal{M}$, 
 $1\leq p,q,r\leq\infty$ with $1+1/r=1/p+1/q$, we have
\begin{align}\label{eq:Young inequality introduction}
\|x\ast y\|_r\leq k\|x\|_p\|y\|_q.
\end{align}
\end{theorem}

In Section \ref{sec:entropy convolution inequality},
 we prove quantum entropic convolution inequalities (qECI) on FN $k$-algebras. Moreover, we characterize the extremizers of qECI when the FN $k$-algebras are from subfactors. 
\begin{theorem}[See Theorem \ref{thm:entropy convolution inequality}]
Let $(\mathcal{M},\tau,\ast,\rho)$ be a FN $k$-algebra. For any positive operators  $x,y\in\mathcal{M}$ with $\|x\|_{1}=\|y\|_1=k^{-1}$,  $0\leq\theta\leq1$, we have
\begin{align}\label{eq:entropy convolution inequality introduction}
H(x\ast y)\geq\theta H( x)+ (1-\theta)H( y),
\end{align}
where $H(x)=\tau(-x\log x)$ is the von Neumann entropy.
\end{theorem}
We introduce the notion of smooth convolution entropy.
 \begin{definition}\label{def:convolution_smooth_entropy}
Let $(\mathcal{M},\tau)$ be a tracial  
von Neumann  algebra with a good $k$-convolution $\ast$.
For any positive operators $x,y\in\mathcal{M}$, $\epsilon,\eta\in [0,1]$ and $p,q\in[1,\infty]$,
the \textbf{smooth convolution  entropy}  is defined by
\begin{align}\label{eq:definition smooth convolution entropy}
H_{\epsilon,\eta}^{p,q}(x\ast y):=\inf\{H(z\ast w):\  z,w\in\mathcal{M},\,z,w\geq0,\ \|x-z\|_p\leq\epsilon,\ \|y-w\|_q\leq\eta\}. 
\end{align}
\end{definition}
For any $x\in\mathcal{M}$, $x\geq0$, the smooth entropy of $x$ with $1\leq p\leq\infty$, $0\leq\epsilon\leq1$ is defined as
\begin{align*}
H_{\epsilon}^p(x):=\sup\{H(y):\ y\in\mathcal{M},\ y\geq0,\ \|y-x\|_p\leq \epsilon \}.
\end{align*}
For a finite dimensional FN $k$-algebra $(\mathcal{M},\tau,\ast,\rho)$,
we call $\tau(I)$ the \textbf{Frobenius-Perron dimension} of $\mathcal{M}$, denoted by $d$. We set $\lambda=\min\{\tau(e):\ \text{$e$ is a 
projection in $\mathcal{M}$}\}$.  In Section \ref{sec:entropy convolution inequality}, we prove the quantum smooth entropic convolution inequality.
\begin{theorem}[See Theorem \ref{thm: convolution smooth entropy inequality}]
Let $(\mathcal{M},\tau,\ast,\rho)$ be a finite dimensional  FN $k$-algebra.
Let $p,q\in[1,\infty]$, 
 $\epsilon,\eta,\in[0,1]$, $\epsilon+\eta\leq\displaystyle\frac{1}{(d+1)(1+k(d+1))}$.
 For any positive operators $x,y\in\mathcal{M}$ with $\|x\|_1=\|y\|_1=k^{-1}$, $0
\leq\theta\leq1$, we have
 \begin{align}\label{eq:quantum smooth convolution entropy inequality introduction}
H_{\epsilon,\eta}^{p,q}(x\ast y)
 \geq&\theta H_{\epsilon}^{p}(x)+(1-\theta)H_{\eta}^{q}(y)-O_{d,\lambda,k}(|\epsilon\log\epsilon|)-O_{d,\lambda,k}(|\eta\log\eta|).
 \end{align}
 \end{theorem}

  The paper is organized as follows: In Section \ref{sec:Comultiplications and convolution von Neumann algebras}, we  prove the positivity 
  of comultiplications on subfactors, which 
provides  analytic obstructions of unitary categorifications  of fusion rings stronger than the one from the positivity of convolution \cite{LPW}.  Section \ref{sec:Implement of Young inequality} is devoted to 
the proof of the quantum Young's inequalities on FN $k$-algebras, which
 unifies the quantum Young's inequalities on subfactor planar algebras and fusion bi-algebras. In the last section, we show the quantum (smooth) entropic convolution inequalities on FN $k$-algebras. Moreover,  the extremizers of qECI are characterized when the FN $k$-algebras are from subfactors.
  
\begin{ac}
Linzhe Huang was supported by YMSC, Tsinghua University.
Zhengwei Liu was supported by NKPs (Grant no. 2020YFA0713000), by Tsinghua University (Grant no. 04200100122) and by Templeton Religion Trust (TRT 159).
Jinsong Wu was supported by NSFC 
(Grant no. 12031004) and a grant from Beijing Institute of Mathematical Sciences and Applications.. 
\end{ac}

\section{Positivity of Comultiplications and Unitary categorification Criterion}\label{sec:Comultiplications and convolution von Neumann algebras}
In this section, we will introduce the notion of comultiplications on von Neumann algebras. We prove the positivity of the comultiplications on subfactor planar algebras, which leads to an  analytic criterion of unitary categorifications of fusion rings stronger than Schur product criterion (See Proposition 9.3 in \cite{LPW}). 
\subsection{Preliminaries}We first recall some basic definitions of von Neumann algebras and non-commutative $L^p$ spaces.
Let $\mathcal{M}$ be a
von Neumann algebra acting on a Hilbert space $\mathcal{H}$ with a faithful normal tracial
positive linear functional $\tau$, see e.g. \cite{KR97}.
 We simply call $\tau$ a trace and the pair $(\mathcal{M},\tau)$ as a tracial von Neumann algebra in the rest of the paper. 
 We refer the readers to
e.g. \cite{Ter81} and \cite{Ter82} for details about non-commutative $L^p$ spaces.

A closed densely defined operator $x$ affiliated with $\mathcal{M}$ is call $\tau$-measurable if for all $\epsilon>0$ there exists a projection $P\in\mathcal{M}$ such that $P\mathcal{H}\subset \mathcal{D}(x)$, and $\tau(I-P)\leq\epsilon$, where $\mathcal{D}(x)$ is the domain of $x$. We denote the set of all $\tau$-measurable closed densely defined operators by $\widetilde{\mathcal{M}}$.  Then $\widetilde{\mathcal{M}}$ is $*$-algebra
with respect to a strong sum, strong product, and adjoint operation. If $x$ is a positive self-adjoint $\tau$-measurable  operator, then $x^{\alpha}\log x$ is also $\tau$-measurable for any $\alpha\in\mathbb{C}$ with positive real part.

The sets
\begin{align*}
U(\epsilon,\eta)=\{ x\in\widetilde{\mathcal{M}}: \exists\ \text{a projection}\ P\in\mathcal{M}\ \text{satsfying}\ P\mathcal{H}\subseteq \mathcal{D}(x),\ \|xP\|\leq\epsilon,\ \tau(I-P)\leq\eta\},
\end{align*}
where $\epsilon,\eta>0$, form a neighborhood basis of $0$  that makes $\widetilde{\mathcal{M}}$ into a
topological vector space. Now $\widetilde{\mathcal{M}}$ is a complete Hausdorff topological $\ast$-algebra and $\mathcal{M}$ is a dense subset of $\widetilde{\mathcal{M}}$.

For any positive self-adjoint operator $x$ affiliated with $\mathcal{M}$, we set
\[\tau(x)=\sup_{n\in\mathbb{N}}\tau(\int_{0}^n tde_t),\] 
 where $x=\int_{0}^\infty tde_t$ is the spectral decomposition of $x$. Then for $p\in[1,\infty)$, the non-commutative $L^p$ space $L^p(\mathcal{M})$ with respect to $\tau_{\mathcal{M}}$
is defined as
\[L^p(\mathcal{M}):=\{x\ \text{closed, densely defined, affiliated with}\ \mathcal{M}:\tau(|x|^p)<\infty  \}.\] 
 The $p$-norm of $x$ is given by
 $\|x\|_p =\tau(|x|^p)^{1/p}$, where $|x|=(x^*x)^{1/2}$. In particular, $\|x\|_{\infty}=\|x\|$, the operator norm.
 We have that $L^p(\mathcal{M})\subseteq\widetilde{\mathcal{M}}$. 
  The following non-commutative  H\"older's inequality will be used frequently in the whole paper.
\begin{proposition}[H\"older's inequality]\label{prop:Holder_inequality}
For any $x,y,z\in\mathcal{M}$, we have
\begin{enumerate}
\item $|\tau(xy)|\leq\|x\|_p\|y\|_q$, where $1\leq p,q\leq\infty$, $\frac{1}{p}+\frac{1}{q}=1$;
\item $|\tau(xyz)|\leq\|x\|_p\|y\|_q\|z\|_r$, where $1\leq p,q,r\leq\infty$, $\frac{1}{p}+\frac{1}{q}+\frac{1}{r}=1$;
\item $\|xy\|_r\leq\|x\|_p\|y\|_q$, where $0< p,q,r\leq\infty$, $\frac{1}{r}=\frac{1}{p}+\frac{1}{q}$.
\end{enumerate}
\begin{proof}
See e.g. Theorems 5.2.2 and 5.2.4 in \cite{Xu}.
\end{proof}
\end{proposition}
\subsection{Comultiplication and Convolution}
The comultiplication and convolution appear as a pair of dual operators on von Neumann algebras. The study of the comultiplication turns out to be more fundamental than the study of the convolution while establishing convolution inequalities on quantum symmetries.
\begin{definition}\label{def:comultiplication}Let $\mathcal{M}$ be a tracial von Neumann algebra.
 A \textbf{k-comultiplication} $\Delta$ is a linear normal map from $\mathcal{M}$ into the spatial tensor product
$\mathcal{M}\overline{\otimes}\mathcal{M}$ with operator norm $\|\Delta\|=\sup_{x\in\mathcal{M}}\frac{\|\Delta(x)\|}{\|x\|}=k$. 
\end{definition}
We shall note that a $k$-comultiplication  may not be a 
 homomorphism, which is different from the case of Kac algebras. 
  A $k$-co-multiplication $\Delta$
induces a \textbf{$k$-convolution} $\ast$  defined as follows:
\begin{align}\label{eq:comul induce convolution}
\left\langle x\ast y,z\right\rangle=\left\langle x\otimes  y,\Delta(z)\right\rangle,\quad \forall x,y,z\in\mathcal{M}.
\end{align}
By H\"older's inequality, we have
\begin{align*}
|\left\langle x\ast y,z \right\rangle|=|\left\langle x\otimes y, \Delta(z)\right\rangle|
\leq\|x\otimes y\|_1\|\Delta(z)\|\leq k\|x\|_1\|y\|_1\|z\|.
\end{align*}
Therefore, \[\| x\ast y\|_1=\sup_{\|z\|=1}\{ |\left\langle x\ast y,z \right\rangle|\}
\leq k\|x\|_1\|y\|_1, \]
which implies that $x\ast y\in L^1(\mathcal{M})$. Consequently, a $k$-convolution satisfies the primary Young's inequality:
\begin{align}\label{eq:elementary Young inequality0}
\|x\ast y\|_1\leq k\|x\|_1\|y\|_1,\quad \forall  x,y\in\mathcal{M}.
\end{align}
 The properties of convolutions inherit the properties of comultiplications naturally.
\begin{proposition}\label{prop:comultiplication positivity imply convolution positivity}
Let $\Delta$ be a $k$-comultiplication and $\ast$ be the induced $k$-convolution. Then for any $x,y,z\in\mathcal{M}$, the following statements holds:
\begin{enumerate}
\item $\Delta\geq0$ implies $x\ast y\geq0$,  $x,y\geq0$;
\item $\Delta(I)=kI$ implies $\tau(x\ast y)=k\tau(x)\tau(y)$;
\item $(\Delta\otimes I)\Delta=(I\otimes \Delta)\Delta$ implies $(x\ast y)\ast z=x\ast (y\ast z)$.
\end{enumerate}
\begin{proof}
(1)
Let $x,y,z\geq0$, $x,y,z\in\mathcal{M}$. We have
\begin{align*}
\left\langle x\ast y,z\right\rangle=\left\langle x\otimes y,\Delta(z)\right\rangle\geq0.
\end{align*}
Therefore, $x\ast y\geq0$.

(2) We have
\begin{align*}
\tau(x\ast y)=\left\langle x\ast y,I\right\rangle=\left\langle x\otimes y,\Delta(I)\right\rangle=\left\langle x\otimes y,kI\right\rangle=k\tau(x)\tau(y).
\end{align*}

(3) For any $x,y,z,w\in\mathcal{M}$, we have
\begin{align*}
\left\langle(x\ast y)\ast z,w\right\rangle &=
\left\langle (x\ast y) \otimes z, \Delta(w)\right\rangle\\&=\left\langle x\otimes y \otimes z, (\Delta\otimes I)\Delta(w)\right\rangle\\
&=\left\langle x\otimes y \otimes z, (I\otimes\Delta)\Delta(w)\right\rangle\\
&=\left\langle x\otimes (y \ast z), \Delta(w)\right\rangle\\
&=\left\langle x\ast (y \ast z), w\right\rangle,
\end{align*}
which implies $(x\ast y)\ast z=x\ast (y \ast z)$.
\end{proof}
\end{proposition}
 
 Note that every linear map $\widetilde{\Delta}$: $\mathcal{M}\to \mathcal{M}\overline{\otimes}\mathcal{M}$ satisfying  $\widetilde{\Delta}\geq0$ and $\widetilde{\Delta}(I)=kI$ has operator norm $k$ due to Russo-Dye Theorem \cite{RD66}:
 every unital positive linear map between two $C^*$-algebras is contractive.
\begin{definition}\label{def:good comultiplication}
A $k$-comultiplication $\Delta$: $\mathcal{M}\to \mathcal{M}\overline{\otimes}\mathcal{M}$ is called a \textbf{good $k$-comultiplication} if $\Delta\geq0$ and $\Delta(I)=kI$. We simply call $\Delta$ a \textbf{good comultiplication} when $k=1$, i.e., $\Delta$ is a unital positive linear map.
\end{definition}
\begin{remark}\label{rem:good comultiplication induce good convolution}
A good $k$-comultiplication $\Delta$ induces a 
\textbf{good $k$-convolution} $\ast$ satisfying the following properties  from Proposition \ref{prop:comultiplication positivity imply convolution positivity} and inequality \eqref{eq:elementary Young inequality0}:
\begin{enumerate}
\item Positivity:
\[ x\ast y\geq0,\quad \forall\ x,y\geq0,\ x,y\in\mathcal{M};\]
\item Primary Young's inequality:
\[\|x\ast y\|_{1}\leq k\|x\|_1\|y\|_1,\quad \forall\ x,y\in\mathcal{M};\]
\item Haar measure:
\[\tau(x\ast y)=k\tau(x)\tau(y),\quad \forall\ x,y\in\mathcal{M}.\]
\end{enumerate}
When $k=1$, we also simply call $\ast$ a \textbf{good convolution}.
\end{remark}
\begin{proposition}\label{prop:Haar imply Positivity}
The Haar measure implies Positivity.
\begin{proof}
For any $x,y\geq0$, $x,y\in\mathcal{M}$, we have
\begin{align*}
\|x\|_1\|y\|_1=\tau(x)\tau(y)=k^{-1}\tau(x\ast y)\leq k^{-1}\|x\ast y\|_1\leq\|x\|_1\|y\|_1.
\end{align*}
 Hence $\tau(x\ast y)=\|x\ast y\|_1$. This implies $x\ast y\geq0$.
\end{proof}
\end{proposition}

The positivity plays a fundamental role in subfactor theory such as in the proof of the remarkable Jones index theorem \cite{Jon83}, and it is formulated as reflection positivity in subfactor planar algebras \cite{Jon99}. 
\subsection{Positivity of Comultiplication and Categorification Criterion}
The positivity of the convolution is called quantum Schur product theorem in subfactor planar algebras (See Theorem 4.1 in \cite{Liu}), which provides an analytic tool in classification of subfactor planar algebras \cite{Liu,BLJ17,LMP15}. The Schur product property on the dual of fusion rings gives an analytic criterion of unitary categorifications of fusion rings (called Schur product criterion, see Proposition 8.3 in \cite{LPW}). In this section, we will show the positivity of the comultiplication on subfactor planar algebras, which leads to a stronger criterion for unitary categorification.

Suppose $\mathscr{P}_{\bullet}$  is a subfactor planar algebra with finite index $\delta^2$, $\delta>0$. Recall that the
convolution on the 2-box spaces is defined as follows:
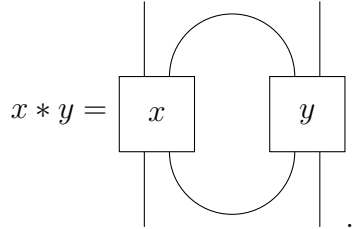
\begin{align*}
\begin{tikzpicture}
\draw(-0.8,0.5)node{$x\ast y=$};
\draw(0,0) --(1,0)--(1,1)--(0,1)--(0,0);
\draw(0.5,0.5) node{$x$};
\draw(2,0) --(3,0)--(3,1)--(2,1)--(2,0);
\draw(2.5,0.5) node{$y$};
\draw (2.333,1) arc (0:180:0.833);
\draw (2.333,0) arc (0:-180:0.833);
\draw(0.333,1)--(0.333,2);
\draw(2.666,1)--(2.666,2);
\draw(0.333,0)--(0.333,-1);
\draw(2.666,0)--(2.666,-1);
\end{tikzpicture}.
\end{align*}
The convolution is also called coproduct (See e.g.  page 61 in \cite{BJ00} and page 11 in \cite{JLW16}). Recall that a biprojection is a projection  with respect to both multiplication and convolution (up to a scalar). We refer to \cite{Bir94,BJ00,BJ03} for more details. The biprojection has been applied to characterize the extremizers of uncertain principles
\cite{JLW16} and quantum Young's inequality \cite{JLW19} and
will appear in 
Theorems \ref{thm:quantum reverse Young 1} and \ref{thm:encon on subfactor} in this paper.

We consider the comultiplication $\Delta$: $\mathscr{P}_{2,\pm}\rightarrow\mathscr{P}_{2,\pm}\otimes\mathscr{P}_{2,\pm}$ as the adjoint operator of $\ast$ with respect to the unnormalized faithful Morkov trace {\rm Tr}:
\begin{align}\label{eq:definition of comultiplication on subfactor}
\left\langle \Delta(z),x\otimes y\right\rangle=\left\langle z,x\ast y\right\rangle
\end{align}
for any $x,y,z\in\mathscr{P}_{2,\pm}$. 
Switching the input disc and output discs of the convolution tangle, we obtain a surface tangle for comultiplication, see \cite{Liu19} for the theory of surface tangles.
\begin{theorem}[Positivity of Comultiplication]\label{thm:positivity of comultiplication on subfactor}
Suppose $\mathscr{P}_{\bullet}$ is a subfactor planar algebra with finite index $\delta^2$, $\delta>0$. Then the comultiplication  $\Delta$: $\mathscr{P}_{2,\pm}\rightarrow\mathscr{P}_{2,\pm}\otimes\mathscr{P}_{2,\pm}$ is positive.
Moreover, if $\mathscr{P}_{\bullet}$ is irreducible then $\Delta$
is a good $\delta^{-1}$-comultiplication.
\begin{proof}
Let {\rm\textbf{E}} be the conditional expectation from $\mathscr{P}_{4,\pm}$ onto $\mathscr{P}_{2,\pm}\otimes \mathscr{P}_{2,\pm}$. From equality \eqref{eq:definition of comultiplication on subfactor}, we have
\begin{align*}
\begin{tikzpicture}
\draw(-1.3,0.5)node{$\Delta(z)= {\rm \textbf{E}}$};
\draw(0,0) --(1.25,0)--(1.25,1)--(0,1)--(0,0);
\draw(0.625,0.5) node{$z$};
\draw(0.1,1)--(-0.5,1.8);
\draw(1.15,1)--(1.75,1.8);
\draw(1.4,1.8) arc (-20:-160:0.83);
\draw(0.1,0)--(-0.5,-0.8);
\draw(1.15,0)--(1.75,-0.8);
\draw(1.4,-0.8)arc (20:160:0.83);
\draw(1.35,0)node{,};
\end{tikzpicture}
\end{align*}
which is a composition (up to a scalar) of a $^*$-homomorphism and
a conditional expectation. Thus, $\Delta$ is positive.
Now suppose $\mathscr{P}_{\bullet}$ is irreducible.
 Note that 
\begin{align*}
\left\langle x\otimes y, \Delta(I)\right\rangle&=
\left\langle x\ast y,I\right\rangle\\
&={\rm Tr}(x\ast y)=\frac{1}{\delta}{\rm Tr}(x){\rm Tr}(y).
\end{align*} 
 Since {\rm Tr} is faithful, so $\Delta(I)=\delta^{-1}I$. Therefore, $\Delta$ is a good $\delta^{-1}$-comultiplication.
\end{proof}
\end{theorem}
\begin{remark}
From  Proposition \ref{prop:comultiplication positivity imply convolution positivity} $(1)$, Theorem \ref{thm:positivity of comultiplication on subfactor} implies Theorem 4.1 in \cite{Liu}.
\end{remark}
\begin{corollary}\label{cor:elementary Young inequality subfactor}
Suppose $\mathscr{P}_{\bullet}$ is an irreducible subfactor planar algebra with finite index $\delta^2$, $\delta>0$.
For any $x,y\in\mathscr{P}_{2,\pm}$, we have
\[\|x\ast y\|_1\leq \frac{\|x\|_1\|y\|_1}{\delta}.\]
\begin{proof}
It follows from Theorem \ref{thm:positivity of comultiplication on subfactor} and Remark \ref{rem:good comultiplication induce good convolution}.
\end{proof}
\end{corollary}
Let $\mathfrak{A}$ be a fusion ring with basis $\{x_1=I,x_2,\ldots,x_n\}$ satisfying the following fusion rules: 
\[x_k x_j=\sum_{i=1}^n N_{k,j}^i x_i,\quad N_{k,j}^i\in\mathbb{N}.\]
Let $M_k=(N_{k,j}^i)_{i,j}$ be the fusion matrix of $x_k$.

\begin{definition}[See page 18 in \cite{LPW}]
We define a linear map $ \Delta_1$:  $\mathfrak{A}\rightarrow\mathfrak{A}\otimes\mathfrak{A}$ such that
\begin{align*}
\Delta_1(x_k)=\frac{1}{d(x_k)} x_k\otimes x_k,\quad \Delta_1(x_k^*)=\Delta_1(x_k)^*.
\end{align*}
Then $\Delta_1$ is a $^*$-preserving linear map.
\end{definition}

\begin{proposition}\label{prop: fusion ring categorification imply positivity}
If a fusion ring $\mathfrak{A}$ admits a unitary categorification, then the linear map $\Delta_1$
 is positive, i.e., $\Delta_1(x)\geq0$, for any  $x\geq0$, $x\in\mathfrak{A}$.
 \begin{proof}
 It follows from Theorem \ref{thm:positivity of comultiplication on subfactor} and Proposition 2.25 in \cite{LPW}.
 \end{proof}
\end{proposition}

\begin{proposition}\label{prop: positive comultiplication matrix}
The linear map $\Delta_1$  is positive if and only if for any $v\in\mathbb{C}^n$
\begin{align}\label{eq:positive comultiplication criterion}
T=\sum_{k=1}^n\frac{v^* M_k v}{\|M_k\|_{\infty}}
M_k\otimes M_k\geq 0.
\end{align}
\begin{proof}
Since every positive operator in $\mathfrak{A}$ has the form $(\sum_{i=1}^n a_i x_i)(\sum_{j=1}^n a_j x_j)^*$, $a_i\in\mathbb{C}$, so
the linear map $\Delta_1$ is positive if and only if
$\Delta_1(\sum_{i=1}^n a_i x_i)(\sum_{j=1}^n a_j x_j)^*$ is positive. Therefore,
\begin{align}\label{eq:positive comultiplication}
\Delta_1\bigg(\sum_{i=1}^n a_i x_i\bigg)\bigg(\sum_{j=1}^n a_j x_j\bigg)^*&=\Delta_1\bigg(\sum_{i=1}^n a_i x_i\bigg)\bigg(\sum_{j=1}^n\overline{a_j} x_{j^*}\bigg)\notag\\
&=\Delta_1\bigg(\sum_{i,j=1}^n a_i \overline{a_j} x_i x_{j^*}\bigg)\notag\\
&=\Delta_1\bigg(\sum_{i,j=1}^n a_i \overline{a_j} \sum_{k=1}^n N_{i,j^*}^k x_k\bigg)\notag\\
&=\sum_{k=1}^n\frac{1}{d(x_k)} \bigg(\sum_{i,j=1}^n a_i \overline{a_j} N_{i,j^*}^k\bigg)x_k\otimes x_k\notag\\
&=\sum_{k=1}^n\frac{1}{d(x_k)}\bigg (\sum_{i,j=1}^n a_i \overline{a_j} N_{k,j}^i\bigg)x_k\otimes x_k\quad \text{Frobenius reciprocity}\notag\\
&=\sum_{k=1}^n\frac{v^* M_k v}{d(x_k)} x_k\otimes x_k\geq0. 
\end{align}
Now we consider the left regular representation of $\mathfrak{A}\otimes\mathfrak{A}$ on $L^2(\mathfrak{A}\otimes\mathfrak{A})$. Then inequality (\ref{eq:positive comultiplication}) is equivalent to inequality (\ref{eq:positive comultiplication criterion}).
\end{proof}
\end{proposition}

\begin{theorem}[Positivity of Comultiplication Criterion]
\label{thm:positivity of comultiplication criterion}
For a fusion ring $\mathfrak{A}$, let $N_{k,j}^i$
be the fusion coefficients and
 $M_k=(N_{k,j}^i)_{i,j}$
be the fusion matrices, $1\leq i,j,k\leq n$.
If $\mathfrak{A}$ admits a unitary categorification, then for any $v\in\mathbb{C}^n$,
\begin{align}\label{eq:positive comultiplication criterion}
T=\sum_{k=1}^n\frac{v^* M_k v}{\|M_k\|_{\infty}}
M_k\otimes M_k\geq 0.
\end{align}
\begin{proof}
It follows from Propositions \ref{prop: fusion ring categorification imply positivity} and \ref{prop: positive comultiplication matrix}.
\end{proof}
\end{theorem}

\begin{remark}For any $v_s\in\mathbb{C}^n$, $s=1,2,3$, inequality \eqref{eq:positive comultiplication criterion} implies
\begin{align*}
\left\langle T( v_2\otimes v_3),v_2\otimes v_3\right\rangle\geq0.
\end{align*}
Therefore, we obtain the non-commutative Schur product criterion (see Proposition 7.3 in \cite{LPW}):
\begin{align}\label{eq:Non Commutative Schur Product Criterion}
\sum_{k=1}^n\frac{1}{\|M_k\|}\prod_{s=1}^3 v_s^* M_k v_s\geq0.
\end{align}
Note that a matrix $M$ acting on $\mathbb{C}^n\otimes \mathbb{C}^n$ such that $\left\langle M v\otimes w,v\otimes w\right\rangle\geq0$ for any $v,w\in\mathbb{C}^n$ does not assure that $M$ is positive-semidefinite. For example \footnote{
We thank S. Palcoux for providing this example.}, let
\begin{align*}
M=\begin{pmatrix}
1&0&0&0\\0&0&1&0\\0&1&0&0\\0&0&0&1
\end{pmatrix}.
\end{align*}
Then $M$ is not positive-semidefinite since the determinant of $M$ is negative. However,
for $v=(v_1,v_2)',\ w=(w_1,w_2)'\in\mathbb{C}^2$, we have
\begin{align*}
\left\langle M v\otimes w,v\otimes w\right\rangle=|v_1|^2|w_1|^2+2Re(v_1 w_2\overline{v_2 w_1})+|v_2|^2|w_2|^2=|\left\langle v,w\right\rangle|^2\geq0.
\end{align*}
This example supports the idea that the positivity of comultiplication criterion is stronger than the non-commutative Schur product criterion.
\end{remark}

Now let $(\mathcal{A},\mathcal{B},d,\tau,\mathcal{F})$ be a fusion bi-algebra (see Definition 2.13 in \cite{LPW}), where $\mathcal{A}$ and $\mathcal{B}$ are finite dimensional  $C^*$-algebras,
$\mathcal{A}$ is commutative, and $\mathcal{F}$ is a unitary from $\mathcal{A}$ onto $\mathcal{B}$ preserving $2$-norm.
Let $\{I=x_1,x_2,\ldots,x_n\}$ be the unique $\mathbb{R}_{\geq0}$-basis of $\mathcal{B}$ such that $\mathcal{F}^{-1}(x_i)$ is a multiple of minimal projection in $\mathcal{A}$ and satisfying the fusion rules: 
\begin{align*}
x_k x_j=\sum_{i=1}^n N_{k,j}^i x_i,\quad N_{k,j}^i\geq0.
\end{align*}
 The unitary $\mathcal{F}$ induces a convolution $\ast$ of $\mathcal{A}$:
\begin{align*}
\mathcal{A}\otimes\mathcal{A}\rightarrow\mathcal{A},\quad x\otimes y\mapsto x\ast y=\mathcal{F}^{-1}(\mathcal{F}(x)\mathcal{F}(y)).
\end{align*}
We define a linear map $\Delta_2$: $\mathcal{A}\rightarrow\mathcal{A}\otimes\mathcal{A}$ as the adjoint operator of $\ast$ with respect to the trace $d$:
\begin{align}\label{eq:definition of comultiplication on fusing bialgebra}
\left\langle \Delta_2(z),x\otimes y\right\rangle=\left\langle z,x\ast y\right\rangle
\end{align}
for any $x,y,z\in\mathcal{A}$.
\begin{proposition}\label{prop:positivity comultiplication fusion bialgebra}
Let $(\mathcal{A},\mathcal{B},d,\tau,\mathcal{F})$ be a fusion bi-algebra. The linear map $\Delta_2$: $\mathcal{A}\to\mathcal{A}\otimes \mathcal{A}$
defined in equality \eqref{eq:definition of comultiplication on fusing bialgebra} is a good comultiplication.
\begin{proof}
For any $x_i$, $x_j$, $x_k$,
we have
\begin{align*}
\left\langle \mathcal{F}^{-1}(x_i)\otimes\mathcal{F}^{-1}(x_j),\Delta_2(\mathcal{F}^{-1}(x_k))\right\rangle&=\left\langle \mathcal{F}^{-1}(x_i)\ast\mathcal{F}^{-1}(x_j),\mathcal{F}^{-1}(x_k)\right\rangle\\
&=\langle\mathcal{F}^{-1}(\sum_{s=1}^n N_{i,j}^s x_s), \mathcal{F}^{-1}(x_k)\rangle\\
&=\sum_{s=1}^n N_{ij}^s d(x_s x_k^*)\\
&=N_{i,j}^k\geq0.
\end{align*}
Since $\mathcal{A}$ is commutative and 
\[ \sum_{i=1}^n d(\mathcal{F}^{-1}(x_i))\mathcal{F}^{-1}(x_i)=I,\] thus $\Delta_2\geq0$. By Proposition 2.5 in \cite{LPW}, we further have
\begin{align*}
\left\langle \mathcal{F}^{-1}(x_i)\otimes\mathcal{F}^{-1}(x_j),\Delta_2(I)\right\rangle&=\left\langle \mathcal{F}^{-1}(x_i)\ast\mathcal{F}^{-1}(x_j),I\right\rangle\\
&=d(\mathcal{F}^{-1}(x_i)\ast\mathcal{F}^{-1}(x_j))=d(\mathcal{F}^{-1}(x_i))d(\mathcal{F}^{-1}(x_j)),
\end{align*}
which implies $\Delta_2$ preserves identity. Therefore, $\Delta_2$ is a good comultiplication.
\end{proof}
 
\end{proposition}

Quantum inequalities were studied  and applied in  \cite{Liu,BLJ17,LMP15,LPW}. In  Sections \ref{sec:Implement of Young inequality} and \ref{sec:entropy convolution inequality}, we will unify various quantum inequalities including quantum Young's inequality, quantum (smooth) entropic convolution  inequality on Frobenius von Neumann $k$-algebras.  These inequalities can be applied as criterions of unitary categorification as well.

\section{Quantum Young's inequality on Frobenius von Neumann  $k$-algebras}\label{sec:Implement of Young inequality}
In this section, we will prove the quantum Young's inequality on Frobenius von Neumann  $k$-algebras, which unifies the quantum Young's inequality on 
subfactor planar algebras \cite{JLW16} and fusion bi-algebras \cite{LPW}. 
\subsection{Frobenius von Neumann $k$-algebras}
To unify quantum inequalities on quantum symmetries, we introduce the notion of Frobenius von Neumann $k$-algebras.
\begin{definition}\label{def:antipode}
Let $(\mathcal{M},\tau)$ be a tracial von Neumann algebra with a $k$-convolution $\ast$, $k>0$. A trace-preserving anti $^*$-isomorphism $\rho$ on $\mathcal{M}$ is called an $\textbf{antipode}$ if 
the Frobenius Reciprocity holds:
\begin{align}\label{eq:Frobenius Reciprocity}
\tau((x\ast y)z)=\tau 
((\rho(z)\ast x)\rho(y)),\quad\forall\ x,y,z\in\mathcal{M}.
\end{align}
\end{definition}

\begin{proposition}\label{lem:rotation preserve pnorm}
 For any $x\in\mathcal{M}$, $0<p\leq\infty$, the antipode $\rho$ satisfies $\|\rho(x)\|_p=\|x\|_p.$
\begin{proof}
For any positive rational number $r=\frac{m}{n}$, $m,n\in\mathbb{N}_+$, we have
\begin{align*}
(|\rho(x)|^r)^n =|\rho(x)|^m=\rho(|x^*|)^m=\rho(|x^*|^m),\quad
\rho(|x^*|^r)^n =\rho(|x^*|^{rn})=\rho(|x^*|^m).
\end{align*}
Since $|\rho(x)|^r$ and $\rho(|x^*|^r)$ are positive operators,
we have $|\rho(x)|^r=\rho(|x^*|^r)$. Therefore,
\begin{align*}
\|\rho(x)\|_r=\tau(|\rho(x)|^r)^{1/r}=\tau(\rho(|x^*|^r))^{1/r}=
\tau(|x^*|^r)^{1/r}=\|x^*\|_r=\|x\|_r.
\end{align*}
Since positive rational numbers are dense in positive real numbers and $\tau$ is continuous with respect to operator norm, we have $\|\rho(x)\|_p=\|x\|_p$ for any $0<p<\infty$. Note that the operator norm $\|x\|=\lim_{n\rightarrow\infty}\tau(|x|^n)^{1/n}$, we have
$\|\rho(x)\|=\|x\|$. Therefore $\|\rho(x)\|_p=\|x\|_p$ for any $0<p\leq\infty$.
\end{proof}
\end{proposition}

\begin{definition}\label{def:vNFA}
Let $(\mathcal{M},\tau)$ be a tracial von Neumann algebra with a good $k$-convolution $\ast$ and an antipode $\rho$. Then we call the quadruple $(\mathcal{M},\tau,\ast,\rho)$ a \textbf{Frobenius  von Neumann $k$-algebra} (\textbf{FN $k$-algebra}). We simply call  $(\mathcal{M},\tau,\ast,\rho)$ a \textbf{Frobenius von Neumann algebra} (\textbf{FN algebra}) when $k=1$.
\end{definition}
\begin{remark}\label{rem:vNF k algebra reduce to vNF algebra}
Let $(\mathcal{M},\tau,\ast,\rho)$ be a FN $k$-algebra. For any $\lambda_1,\lambda_2>0$, define
\[\tau_{\lambda_1}(x)=\lambda_1^{-1}\tau(x),\quad x\ast_{\lambda_2}y=\lambda_2^{-1}x\ast y,\quad\forall x,y\in\mathcal{M}.\]
Then $(\mathcal{M},\tau_{\lambda_1},\ast_{\lambda_2},\rho)$ is a FN $\lambda_1 k/\lambda_2$-algebra. In particular,  $(\mathcal{M},\tau_{\lambda_1},\ast_{\lambda_2},\rho)$ is a FN algebra  when $\lambda_1/\lambda_2=1/k$.
\end{remark}
\begin{example}\label{ex:vNF algebra from subfactor}
Let $\mathscr{P}_{2,\pm}$ be the 2-box space of an irreducible subfactor planar algebra with finite index $\delta^2$, $\delta>0$. For any $x\in\mathscr{P}_{2,\pm}$, let $\rho(x)$ be the contragredient of $x$.
Then $\rho$ is an antipode on 
$\mathscr{P}_{2,\pm}$  by Lemma 3.4 in \cite{JLW16}. The quadruple $(\mathscr{P}_{2,\pm},{\rm Tr},\ast,\rho)$ is a FN $\delta^{-1}$-algebra.
\end{example}

\begin{example}\label{ex:vNF algebra from fusion bi-algebra}
Let $(\mathcal{A},\mathcal{B},d,\tau,\mathcal{F})$
be a fusion bi-algebra. For any $x\in\mathcal{A}$,
 let $\rho(x)=J(x)^*$,  where $J=\mathcal{F}^{-1}(\mathcal{F}(x)^*)$ is an anti-linear, $\ast$-isomorphism on $\mathcal{A}$ (See Definition 2.12 in \cite{LPW}). Then $\rho$ is an antipode  on $\mathcal{A}$ by Proposition 2.21 in \cite{LPW}. 
 The quadruple $(\mathcal{A},d,\ast,\rho)$ is a FN algebra.
\end{example}

\subsection{Quantum Young's inequality }\label{subsec:Young inequality on cvna} 
In this section, we will prove the quantum Young's inequality on Frobenius von Neumann  $k$-algebras.
The interpolation theorem for bounded linear maps between two tracial von Neumann algebras would be helpful in the proof.
\begin{proposition}[Interpolation Theorem, see e.g.  \cite{Kos} ]\label{prop:Interpolation Theorem}
Let $\mathcal{M}$ and $\mathcal{N}$ be two von Neumann algebras with traces $\tau_{1}$ and $\tau_{2}$ respectively. Suppose $T:\mathcal{M}\rightarrow\mathcal{N}$ is a linear map.
If 
\[\|Tx\|_{p_1,\tau_1}\leq K_1\|x\|_{q_1,\tau_1},\quad \|Tx\|_{p_2,\tau_2}\leq K_2\|x\|_{q_2,\tau_2}, \]
then
\[\|Tx\|_{p_{\theta}}\leq K_1^{1-\theta}K_2^{\theta}\|x\|_{q_{\theta}},\]
where $\frac{1}{p_{\theta}}=\frac{1-\theta}{p_1}+\frac{\theta}{p_2}$, $\frac{1}{q_{\theta}}=\frac{1-\theta}{q_1}+\frac{\theta}{q_2}$, 
$1\leq p_1,q_1,p_2,q_2\leq\infty$, $0\leq\theta\leq1$.
\end{proposition}
For a FN $k$-algebra $(\mathcal{M},\tau,\ast,\rho)$, 
we shall note from the definition of good $k$-convolution (see Remark \ref{rem:good comultiplication induce good convolution}) that  the primary Young's inequality holds :
\begin{align}\label{eq:elementary Young inequality}
\|x\ast y\|_{1}\leq k\|x\|_1\|y\|_1
\end{align}
for any $x,y\in\mathcal{M}$.

\begin{lemma}\label{lem:Young inequality infty1infty}
For any $x,y\in\mathcal{M}$, 
we have 
\[\|x\ast y\|_{\infty}\leq k\|x\|_1\|y\|_{\infty}.\]
\begin{proof}
We have
\begin{align*}
\|x\ast y\|_{\infty}=\sup_{\|z\|_{1}=1}|\tau((x\ast y)z)|&=\sup_{\|z\|_{1}=1}
|\tau((\rho(z)\ast x)\rho(y))|\\
&\leq \sup_{\|z\|_{1}=1}\|\rho(z)\ast x \|_1\|\rho(y)\|_{\infty}\leq k\|x\|_1\|y\|_{\infty}.
\end{align*}
The first inequality is H\"older's inequality and the second uses primary Young's inequality and Lemma \ref{lem:rotation preserve pnorm}.
\end{proof}
\end{lemma}

\begin{lemma}\label{lem:Young inequality inftyinfty1}
For any $x,y\in\mathcal{M}$, 
we have 
\[\|x\ast y\|_{\infty}\leq k\|x\|_{\infty}\|y\|_{1}.\]
\begin{proof}
We have
\begin{align*}
\|x\ast y\|_{\infty}=\sup_{\|z\|_{1}=1}|\tau((x\ast y)z)|&=\sup_{\|z\|_{1}=1}
|\tau((\rho(z)\ast x)\rho(y))|\\
&\leq \sup_{\|z\|_{1}=1}\|\rho(z)\ast x \|_{\infty}\|\rho(y)\|_{1}\leq k\|x\|_{\infty}\|y\|_{1}.
\end{align*}
The first inequality is H\"older's inequality and the second uses Lemmas \ref{lem:Young inequality infty1infty} and \ref{lem:rotation preserve pnorm}.
\end{proof}
\end{lemma}

\begin{lemma}\label{lem:Young inequality pp1 p1p}
For any $x,y\in\mathcal{M}$, $1\leq p\leq\infty$, we have \[
\|x\ast y\|_p\leq k\|x\|_1\|y\|_p,\quad \|x\ast y\|_p\leq k\|x\|_p\|y\|_1.\]
\begin{proof}
It follows from primary Young's inequality,  Lemmas
\ref{lem:Young inequality infty1infty} and \ref{lem:Young inequality inftyinfty1}, and Proposition \ref{prop:Interpolation Theorem}.
\end{proof}
\end{lemma}

\begin{lemma}\label{lem:young inequality infty p q}
For any $x,y\in\mathcal{M}$, $1\leq p\leq\infty$,
$\frac{1}{p}+\frac{1}{q}=1$,
we have \[\|x\ast y\|_{\infty}\leq k\|x\|_p\|y\|_q.\]
\begin{proof}
We have
\begin{align*}
\|x\ast y\|_{\infty}=\sup_{\|z\|_1=1}|\tau((x\ast y)z)|
&=\sup_{\|z\|_1=1}|\tau((\rho(z)\ast x)\rho(y))|
\\
&\leq \sup_{\|z\|_1=1}\|\rho(z)\ast x\|_p\|\rho(y)\|_q\leq
k\|x\|_p\|y\|_q.
\end{align*}
The second inequality uses Lemma \ref{lem:Young inequality pp1 p1p}.
\end{proof}
\end{lemma}

\begin{theorem}[Quantum Young's inequality]\label{thm:quantum Young inequality}
Let $(\mathcal{M},\tau,\ast,\rho)$ be a FN $k$-algebra. 
 For any
$x,y \in\mathcal{M}$,  $1\leq p,q,r\leq\infty$ with $1/r+1=1/p+1/q$, we have
\begin{align}\label{eq:quantum Young inequality}
\|x\ast y\|_r\leq k\|x\|_p\|y\|_q.
\end{align}
\begin{proof}
It follows from Lemma \ref{lem:Young inequality pp1 p1p}, Lemma \ref{lem:young inequality infty p q} and Proposition \ref{prop:Interpolation Theorem}.
\end{proof}
\end{theorem}

\begin{remark}We have the following statements:
\begin{enumerate}
\item Theorem \ref{thm:quantum Young inequality} implies Theorem 4.13 in \cite{JLW16}.
\item Theorem \ref{thm:quantum Young inequality}  implies Theorem 5.11 in \cite{LPW}.
\end{enumerate}
\end{remark}

\begin{theorem}
Let $(\mathcal{M},\tau_{\mathcal{M}},\ast,\rho)$ be a FN $k$-algebra. For fixed $x_i,y_i\in\mathcal{M}$,  $i=1,2,\ldots,n$,
$1\leq p,q,r\leq \infty$ with $1+1/r=1/p+1/q$, there exists  $t_0\in[0,1]$ such that
\[\bigg\|\sum_{i=1}^n x_i\ast y_i\bigg\|_r\leq k\bigg\|\sum_{j=1}^n  
e^{2\pi ijt_0}x_j\bigg\|_p\bigg\|\sum_{k=1}^n e^{-2\pi ikt_0}y_k\bigg\|_q.\]
\begin{proof}
Note that 
\[\int_{0}^1 e^{2\pi i(j-k)t}dt=\delta_{j,k},\quad j,k\in\mathbb{N}.\]
We have
\begin{align*}
\sum_{i=1}^n x_i\ast y_i&= 
\int_{0}^1\sum_{1\leq j,k\leq n}
e^{2\pi i (j-k)t}x_j\ast y_k dt\\
&=\int_{0}^1 (\sum_{j=1}^n e^{2\pi ijt} x_j)\ast (\sum_{k=1}^n e^{-2\pi ikt} y_k)dt.
\end{align*}
By Minkowski inequality, 
\begin{align*}
\bigg\|\sum_{i=1}^n x_i\ast y_i\bigg\|_r&=\bigg\| 
\int_{0}^1 (\sum_{j=1}^n e^{2\pi ijt} x_j)\ast (\sum_{k=1}^n e^{-2\pi ikt} y_k)dt\bigg\|_r\\
&\leq
\int_{0}^1\bigg\| (\sum_{j=1}^n e^{2\pi ijt} x_j)\ast (\sum_{k=1}^n e^{-2\pi ikt} y_k)\bigg\|_rdt\\
&\leq k\int_{0}^1 \bigg\|\sum_{j=1}^n e^{2\pi ijt} x_j\bigg\|_p\bigg\|\sum_{k=1}^n e^{-2\pi ikt} y_k 
\bigg\|_q dt\quad \text{quantum Young's inequality}\\
&\leq \sup_{t\in[0,1]} k\bigg\|\sum_{j=1}^n e^{2\pi ijt} x_j\bigg\|_p\bigg\|\sum_{k=1}^n e^{-2\pi ikt} y_k\bigg\|_q.
\end{align*}
Since $[0,1]$ is compact, there exists some $t_0$
to obtain the supremum.
\end{proof}
\end{theorem}

\subsection{Quantum reverse Young's inequality}\label{subsec:Reverse Young inequality}
In the end of this section, we will give two versions of  quantum reverse Young's inequality on subfactor planar algebras.
\begin{theorem}[Quantum reverse Young's inequality (1)]\label{thm:quantum reverse Young 1}
Let $\mathscr{P}_{\bullet}$ be an irreducible planar algebra with index $\delta^2$, $\delta>0$. For any positive operators $x,y\in\mathscr{P}_{2,\pm}$, $0<r,s,t\leq1$ with $1+1/r=1/s+1/t$, we have
\begin{align}\label{eq:reverse Young inequality1}
\|x\ast y\|_r\geq\delta^{1-2/r}\|x\|_s\|y\|_t.
\end{align}
\begin{proof}
Let $p=id\otimes e_1$ be the biprojection in $\mathscr{P}_{4,\pm}$. Define a map: $
\mathscr{P}_{4,\pm}\to p\mathscr{P}_{4,\pm}p,\ x\mapsto pxp $, which is a unital positive linear map when considering $p$ as the identity of the $C^*$- algebra $p\mathscr{P}_{4,\pm}p$. Note that $f(x)=x^r$, $0<r\leq1$, is an operator-concave function, we have (See e.g. Theorem 2.1 in \cite{Cho74})
\begin{align}\label{eq:majority}
px^rp\leq (pxp)^r,\quad x\geq0.
\end{align}
We have
\begin{align*}
\|x\ast y\|_r&=\|(x\ast y)^r\|_1^{1/r}\\
&=\delta\|(px\otimes y p)^r\|_1^{1/r}\\
&\geq \delta\|p x^r\otimes y^r p\|_1^{1/r}\quad\text{inequality \eqref{eq:majority}}\\
&=\delta^{1-1/r}\|x^r\ast y^r\|_1^{1/r}\\
&=\delta^{1-2/r}\|x\|_r\|y\|_r\\
&\geq \delta^{1-2/r}\|x\|_s\|y\|_t.
\end{align*}
Thus, we obtain the inequality \eqref{eq:reverse Young inequality1}.
\end{proof}
\end{theorem}
\begin{remark}
Suppose
$\|x\ast y\|_r=\delta^{1-2/r}\|x\|_s\|y\|_t$ for some $0<r<1$. Then $\|x\|_r=\|x\|_s$ and $\|y\|_r=\|x\|_s$ from the proof, which implies $x,y$ are both trace-one projections. Note that (See Proposition 4.7)
\[\mathcal{S}(x\ast y)\leq \mathcal{S}(x)\mathcal{S}(y)=1, \]
where $\mathcal{S}(x)={\rm Tr}(\mathcal{R}(x))$, {\rm Tr} is the unnormalized Markov trace,
$\mathcal{R}(x)$ is the range projection of $x$.
Thus, $\mathcal{S}(x\ast y)=1$ and  $\delta x\ast y$
is a trace-one projection. We have $\delta^{-1}=\|x\ast y\|_r>\delta^{1-2/r}=\delta^{1-2/r}\|x\|_s\|y\|_t$. So inequality \eqref{eq:reverse Young inequality1} is not sharp.
\end{remark}
\begin{conjecture}
Let $\mathscr{P}_{\bullet}$ be an irreducible planar algebra with index $\delta^2$, $\delta>0$. For any positive operators $x,y\in\mathscr{P}_{2,\pm}$, $0<r,s,t\leq1$ with $1+1/r=1/s+1/t$, the following sharp inequality holds:
\begin{align}\label{eq:sharp reverse Young ineqaulity}
\|x\ast y\|_r\geq\delta^{-1}\|x\|_s\|y\|_t.
\end{align}
\end{conjecture}
\begin{remark}Suppose inequality \eqref{eq:sharp reverse Young ineqaulity} holds.
Take $t=1$, then $s=r$. We have
\begin{align*}
\lim_{r\rightarrow0^+}\|x\ast 
y\|_r^r\geq\lim_{r\rightarrow0^+}\delta^{-r} \|x\|_r^r\|y\|_1^r.
\end{align*}
It implies the following sum set estimate:
\begin{align}\label{eq:sum set estimate}
\mathcal{S}(x\ast y)\geq \mathcal{S}(x).
\end{align}
See also Theorem 4.1 in \cite{JLW19}.
\end{remark}
To study the sharp inequality, we introduce another version of quantum reverse Young's inequality.
Let $(\mathcal{M},\tau,\ast,\rho)$ be a finite dimensional FN $k$-algebra. We set \[\lambda=\min\{\tau(e):\ e\  \text{is a projection in}\ \mathcal{M}\}.\] 

\begin{theorem}
Let  $(\mathcal{M},\tau,\ast,\rho)$ be a finite dimensional FN $k$-algebra. 
For any positive operators $x,y\in\mathcal{M}$, 
$0<r,s,t\leq1$ with $1+1/r=1/s+1/t$, we have
\begin{align}\label{eq:quantum reverse young inequality 2}
 \|x^r\ast y^r \|_r\geq \lambda^{1/r-r}k\|x\|_t^r\| y\|_s^r.
\end{align}
\begin{proof}
By Remark \ref{rem:vNF k algebra reduce to vNF algebra}, 
$(\mathcal{M},\tau_{\lambda},\ast_{\lambda},\rho)$ is also a FN $k$-algebra. Define
\[\|x\|_{\lambda,p}=\tau_{\lambda}(|x|^p)^{1/p}, \quad 0<p\leq\infty.\]
We have
\[
\|x\|_{\lambda,\infty}\leq \|x\|_{\lambda, p},\quad x\in\mathcal{M},\quad 1\leq p\leq \infty.\]
By Positivity of ``$\ast_\lambda $",
\[x^t\ast_\lambda y^s\leq (\|x\|_{\lambda,\infty}^{t-r}x^r)\ast_\lambda (\|y\|_{\lambda,\infty}^{s-r} y^r )\leq \|x\|_{\lambda,t}^{t-r}\|y\|_{\lambda,s}^{s-r} x^r\ast_\lambda y^r.\]
By Haar measure property,
\[ \|x\|_{\lambda,t}^r\| y\|_{\lambda,s}^r =k^{-1}\|x\|_{\lambda,t}^{r-t}\| y\|_{\lambda,s}^{r-s}\tau_{\lambda}(x^t\ast_\lambda y^s).\]
Therefore,
\begin{align*}
\|x\|_{\lambda,t}^r\| y\|_{\lambda,s}^r 
&\leq k^{-1}\tau_\lambda(x^r\ast_\lambda y^r)\\
&=k^{-1}\tau_\lambda((x^r\ast_\lambda y^r)^r(x^r\ast_\lambda y^r)^{1-r} )\\
&\leq k^{-1}\|(x^r\ast_\lambda y^r)^r\|_{\lambda,1}\|(x^r\ast_\lambda y^r)^{1-r}\|_{\lambda,\infty}\quad \text{H\"older's inequality}\\
&=k^{-1}\|x^r\ast_\lambda y^r\|_{\lambda,r}^r\|x^r\ast_\lambda y^r\|_{\lambda,\infty}^{1-r}\\
&\leq k^{-1}\|x^r\ast_\lambda y^r\|_{\lambda,r}^r \|x^r\ast_\lambda y^r\|_{\lambda,1/r}^{1-r}\\
&\leq k^{-r}\|x^r\ast_\lambda y^r\|_{\lambda,r}^r \|x^r\|_{\lambda,t/r}^{1-r}\|y^r\|_{\lambda,s/r}^{1-r}\quad \text{quantum Young's inequality}
\\
&=k^{-r}
\|x^r\ast_\lambda y^r\|_{\lambda,r}^r \|x\|_{\lambda,t}^{r(1-r)}\|y\|_{\lambda,s}^{r(1-r)}.
\end{align*}
Thus,
\[ 
\|x^r\ast_\lambda y^r\|_{\lambda,r}\geq k\|x\|_{\lambda,t}^r\| y\|_{\lambda,s}^r.\]
It is equivalent to 
\[ \|x^r\ast y^r \|_r\geq \lambda^{1/r-r}k\|x\|_t^r\| y\|_s^r.\]
\end{proof}
\end{theorem}

\begin{corollary}[Quantum reverse Young's inequality (2)]
Let $\mathscr{P}_{\bullet}$ be an irreducible planar algebra with index $\delta^2$, $\delta>0$. For any positive operators $x,y\in\mathscr{P}_{2,\pm}$, $0<r,s,t\leq1$ with $1+1/r=1/s+1/t$, we have
\begin{align}\label{eq:reverse Young inequality on subfactor 2}
\|x^r\ast y^r \|_r\geq \delta^{-1}\|x\|_t^r\| y\|_s^r.
\end{align}
\end{corollary}
\begin{remark}
Take $x=y=e_1$, the Jones projection. Then $\|x^r\ast y^r \|_r=\delta^{-1}\|x\|_t^r\| y\|_s^r$. So inequality \eqref{eq:reverse Young inequality on subfactor 2} is sharp.
\end{remark}
\begin{theorem}[Sum set estimate]\label{thm:sum set estimate}
Let $(\mathcal{M},\tau,\ast,\rho)$ be a  FN $k$-algebra.  For any $x,y\in\mathcal{M}$, we have
\[\mathcal{S}(\mathcal{R}(x)\ast\mathcal{R}(y))\geq\max\{\mathcal{S}(x),\mathcal{S}(y)\}.\]
\begin{proof}
We have
\begin{align*}
\mathcal{S}(x)\mathcal{S}(y)&=\|\mathcal{R}(x)\|_1\|\mathcal{R}(y)\|_1\\
&=k^{-1}\|\mathcal{R}(x)\ast\mathcal{R}(y)\|_1\quad \text{Haar measure}\\
&\leq k^{-1}\|\mathcal{R}(\mathcal{R}(x)\ast\mathcal{R}(y))\|_2\|\mathcal{R}(x)\ast\mathcal{R}(y)\|_2\quad \text{H\"older's inequality}\\
&\leq\mathcal{S}(\mathcal{R}(x)\ast\mathcal{R}(y))^{1/2}\|\mathcal{R}(x)\|_1\|\mathcal{R}(y)\|_2\quad \text{quantum Young's inequality}\\
&=\mathcal{S}(\mathcal{R}(x)\ast\mathcal{R}(y))^{1/2}\mathcal{S}(x)\mathcal{S}(y)^{1/2}.
\end{align*}
Therefore, $\mathcal{S}(\mathcal{R}(x)\ast\mathcal{R}(y))\geq \mathcal{S}(y)$. Similarly, $\mathcal{S}(\mathcal{R}(x)\ast\mathcal{R}(y))\geq \mathcal{S}(x)$.
\end{proof}
\end{theorem}
\begin{remark}We have that
\begin{enumerate}
\item Theorem \ref{thm:sum set estimate} implies Theorem 4.1 in \cite{JLW19};
\item  Theorem \ref{thm:sum set estimate} implies Theorem 5.19 in \cite{LPW}.
\end{enumerate}
\end{remark}

\section{Quantum Entropic convolution Inequality}\label{sec:entropy convolution inequality}
In this section, we will  prove the quantum entropic convolution inequality on Frobenius von Neumann  $k$-algebras and  characterize the extremiziers of the qECI on subfactor planar algebras. Moreover, we 
introduce the notion of (smooth) convolution entropy and  establish the quantum smooth entropic convolution inequality. 
\subsection{Quantum Entropic Convolution Inequality}\label{subsec:entropy convolution inequality}
\begin{theorem}[Quantum entropic convolution inequality]\label{thm:entropy convolution inequality}
Let $(\mathcal{M},\tau,\ast,\rho)$ be a  FN $k$-algebra. 
For any positive operators $x,y\in\mathcal{M}$ with $\|x\|_1=\|y\|_1=k^{-1}$, $0\leq \theta\leq1$,
we have
\begin{align}\label{eq:qECI}
H(x\ast y)\geq\theta H(x)+(1-\theta)H(y),
\end{align}
 where $H(x)=\tau(-x\log x)$ is the von Neumann entropy.
\begin{proof}
Let $0\leq\theta\leq1$.
For any $r\geq1$, define
\[p=\frac{r}{1-\theta+\theta r},\quad q= \frac{r}{(1-\theta)r+\theta}.\]
Then $1+1/r=1/p+1/q$. Define
\[f(r)=\|x\ast y\|_r-k\|x\|_p\|y\|_q. \]
Then $f(1)=\|x\ast y\|_1-k\|x\|_1\|y\|_1=0$ by Haar measure property and $f(r)\leq0$ by quantum  Young's inequality \eqref{eq:quantum Young inequality}. So $f'(1)\leq0$. 
Since
\begin{align*}
\frac{d}{dr}\tau(x^r)=\tau(x^r\log x),
\end{align*}
we have
\begin{align*}
\frac{d}{dr}\|x\|_r\bigg|_{r=1}=\frac{d}{dr}\exp(\frac{1}{r}\log \tau(x^r))\bigg|_{r=1}=-H(x)-\tau(x)\log \tau(x).
\end{align*}
Therefore,
\begin{align*}
f'(1)=&\frac{d}{dr}\|x\ast y\|_r \bigg|_{r=1}-k \frac{d}{dr}\|x\|_p\bigg|_{r=1}\|y\|_1-k\|x\|_1\frac{d}{dr}\|y\|_q \bigg|_{r=1}\leq0,
\end{align*}
which implies
\[H(x\ast y)\geq (1-\theta)H(x)+\theta H(y).\]
Take $\theta=0,1$ respectively. We could obtain the inequality \eqref{eq:qECI}.
\end{proof}
\end{theorem}
Note that a good $k$-convolution $\ast$ may not satisfy associative law.
\begin{corollary} For any  positive operators $x_i\in\mathcal{M}$ with $\|x_i\|_1=k^{-1}$, $0\leq\lambda_i\leq1$, $i=1,2,3,$
 we have
\begin{align*}
H((x_1\ast x_2)\ast x_3)\geq \sum_{i=1}^3\lambda_i H(x_i)
\end{align*}
and
\begin{align*}
H(x_1\ast (x_2\ast x_3))\geq \sum_{i=1}^3\lambda_i H(x_i)
\end{align*}
For any $n\geq4$, we also have the same argument.
\end{corollary}

We  characterize the extremizers of qECI on subfactor planar algebras.
\begin{theorem}\label{thm:encon on subfactor}
Suppose $\mathscr{P}_{\bullet}$ is an irreducible subfactor planar algebra with finite index $\delta^{2}$, $\delta>0$.
For any positive operators $x, y\in \mathscr{P}_{2, \pm}$ with $\|x\|_1=\|y\|_1=\delta$, $0\leq \theta\leq 1$, we have
\begin{align}\label{eq:qECI subfactor}
H(x*y) \geq (1-\theta) H(x)+ \theta H(y).
\end{align}
Moreover, the equality holds
 for some $0<\theta<1$ if and only if $x, y$ are multiples of right shifts of biprojections such that $\mathcal{R}(\mathcal{F}^{-1}(x)^*)=\mathcal{R}(\mathcal{F}^{-1}(y))$.
\end{theorem}
\begin{proof}
Applying Theorem \ref{thm:entropy convolution inequality} to the   FN  $\delta^{-1}$-algebra $(\mathscr{P}_{2,\pm},{\rm Tr},\ast,\rho)$, we could obtain the inequality.
We now assume that \[H(x*y)= (1-\theta) H(x)+ \theta  H(y)\] for some $0<\theta<1$.
Let 
\begin{align}\label{eq:encon1}
g(z)={\rm Tr}\left[\left(x^{(1-\theta) z+\theta}*y^{\theta z+1-\theta}\right)(x*y)^{1-z}\right], \quad 0\leq \text{Re}(z)\leq 1.
\end{align}
Then $g(1)=\delta$.
By H\"older's inequality and quantum Young's inequality, we have that 
\begin{equation}\label{eq:encon2}
\begin{aligned}
|g(z)|\leq & \big\|x^{(1-\theta) z+\theta}*y^{\theta z+1-\theta}\big\|_{\frac{1}{\text{Re}(z)}} \big\|(x*y)^{1-z}\big\|_{\frac{1}{1-\text{Re}(z)}}\\
\leq &\frac{1}{\delta} \big \|x^{(1-\theta) z+\theta}\big\|_{\frac{1}{(1-\theta)\text{Re}(z)+\theta}}\big\|y^{\theta z+1-\theta}\big\|_{\frac{1}{\theta \text{Re}(z)+1-\theta}}\delta^{1-\text{Re}(z)}\\
=& \delta.
\end{aligned}
\end{equation}
Differentiating $g(z)$ with respect to $z$, we obtain that
\begin{align*}
g'(z)=(1-\theta)&{\rm Tr}\big[\big(x^{(1-\theta) z+\theta}\log x *y^{\theta z+1-\theta}\big)\big(x*y\big)^{1-z}\big]\\
+\theta &{\rm Tr}\big[\big(x^{(1-\theta) z+\theta}*y^{\theta z+1-\theta}\log y\big)\big(x*y\big)^{1-z}\big]\\
-&{\rm Tr}\big[\big(x^{(1-\theta) z+\theta}*y^{\theta z+1-\theta}\big)\big(x*y\big)^{1-z}\log (x*y)\big],
\end{align*}
i.e.
\begin{align*}
g'(1)=-(1-\theta)H(x)-\theta H(y)+H(x*y)=0.
\end{align*}
By Proposition 6.3 in \cite{JLW16}, we have that $g(z)\equiv \delta$ for $0\leq \text{Re}(z)\leq 1$.
Hence
\begin{align*}
\left\|x^{(1-\theta) z+\theta}*y^{\theta z+1-\theta}\right\|_{\frac{1}{\text{Re}(z)}}
=\frac{1}{\delta} \left\|x^{(1-\theta) z+\theta}\right\|_{\frac{1}{(1-\theta)\text{Re}(z)+\theta}}  \left\|y^{\theta z+1-\theta}\right\|_{\frac{1}{\theta \text{Re}(z) +1-\theta}}.
\end{align*}
By Theorem 1.3 in \cite{JLW19} and $x,y\geq0$, we have that $x, y$ are multiples of right shifts of biprojections such that \[\mathcal{R}(\mathcal{F}^{-1}(x)^*)=\mathcal{R}(\mathcal{F}(y)).\]

Conversely, assume $x, y$ are multiples of right shifts of biprojections such that $\mathcal{R}(\mathcal{F}^{-1}(x)^*)=\mathcal{R}(\mathcal{F}^{-1}(y))$. Then $x\ast y$ is a multiple of a projection.
By Theorem 1.3 in \cite{JLW19}, we have that inequality (\ref{eq:encon2}) becomes equality. 
Hence $g(z)$ is constant and $g'(1)=0$ and $H(x*y)= (1-\theta) H(x)+ \theta  H(y)$.
\end{proof}

We have partial characterization for $x,y$ when $H(x\ast y)=H(x)$.
\begin{proposition} 
Suppose $\mathscr{P}_{\bullet}$ is an irreducible subfactor planar algebra with finite index $\delta^{2}$, $\delta>0$.
Let $x, y\in \mathscr{P}_{2, \pm}$, $x,y\geq0$, with $\|x\|_1=\|y\|_1=\delta$, $0\leq \theta\leq 1$. Suppose $x$ is a multiple of a projection, then
$H(x\ast y)=H(x)$ if and only if $y$ and $x\ast y$ are  
multiple of projections.
\begin{proof}
Let $y=\sum_{i=1}^n \alpha_i \displaystyle\frac{\delta P_i}{{\rm Tr}(P_i)}$ be the spectral decomposition of $y$, where $P_i$ are projections, $0\leq\alpha_i\leq1$. Then $\sum_{i=1}^n\alpha_i=1$. From the concavity of entropy, we have
\begin{align*}
H(x\ast y)&=H(x\ast \sum_{i=1}^n \alpha_i \frac{\delta P_i}{{\rm Tr(P_i)}})\\
&=H(\sum_{i=1}^n\alpha_i x\ast \frac{\delta P_i}{{\rm Tr(P_i)}})\geq \sum_{i=1}^n \alpha_i H(x\ast  \frac{\delta P_i}{{\rm Tr(P_i)}})\\
&\geq \sum_{i=1}^n\alpha_i H(x)\quad\text{Theorem \ref{thm:encon on subfactor}}\\
&=H(x).
\end{align*}
From $H(x\ast y)=H(x)$, we have $y=\displaystyle\frac{\delta P_{i_0}}{{\rm Tr}(P_{i_0})}$ for some $i_0$.
Let
\begin{align}\label{eq:encon3}
g(z)={\rm Tr}\left[\left(x^z*y\right)(x*y)^{1-z}\right], \quad 0\leq \text{Re}(z)\leq 1.
\end{align}
Then $g(1)=\delta$.
By H\"older's inequality and quantum Young's inequality, we have that 
\begin{equation}\label{eq:encon4}
\begin{aligned}
|g(z)|\leq & \big\|x^z*y\big\|_{\frac{1}{\text{Re}(z)}} \big\|(x*y)^{1-z}\big\|_{\frac{1}{1-\text{Re}(z)}}\\
\leq &\frac{1}{\delta} \big \|x^z\big\|_{\frac{1}{\text{Re}(z)}}\big\|y\big\|_{1}\delta^{1-\text{Re}(z)}\\
=& \delta.
\end{aligned}
\end{equation}
Differentiating $g(z)$ with respect to $z$, we obtain that
\begin{align*}
g'(1)=-H(x)+H(x\ast y)=0.
\end{align*}
By Proposition 6.3 in \cite{JLW16}, we have that $g(z)\equiv \delta$ for $0\leq \text{Re}(z)\leq 1$. Thus inequality \eqref{eq:encon4} becomes equality. We have
\begin{align}\label{eq:encon5}
\big\|x^z*y\big\|_{\frac{1}{\text{Re}(z)}} =\frac{1}{\delta} \big \|x^z\big\|_{\frac{1}{\text{Re}(z)}}\big\|y\big\|_{1}.
\end{align}
Note that $x$ is a multiple of some projection $P$.
 Thus
\begin{align*}
\|Q\ast P_{i_0}\|_{\frac{1}{\text{Re}(z)}}=\frac{1}{\delta}\|Q\|_{_{\frac{1}{\text{Re}(z)}}}\|P_{i_0}\|_1.
\end{align*}
By Theorem 1.5 in \cite{JLW19}, we have $Q\ast P_{i_0}$ is a multiple of projection. Therefore, $x\ast y$ is a multiple of projection.

Conversely, assume that $y$ and $x\ast y$ are multiple of projections. Then equality \eqref{eq:encon5} holds, which implies inequality \eqref{eq:encon4} becomes equality. Thus $g(z)\equiv \delta$ and $g(1)'=0$. We have that $H(x\ast y)=H(x)$.
\end{proof}
 \end{proposition}

Next, we will see a class of good convolutions such that the quantum entropic convolution inequality holds while the quantum Young's inequality does not hold.
\begin{example}\label{ex:vnwca from matrix}
Let $\mathcal{M}=M_n(\mathbb{C})$, $\tau=${\rm Tr}, the unnormalized trace of matrices. Let $U$ be a unitary matrix in $M_n(\mathbb{C})\otimes M_n(\mathbb{C})$, define 
\[\Delta:M_n(\mathbb{C})\rightarrow M_n(\mathbb{C})\otimes M_n(\mathbb{C}):\quad x\mapsto U^* (x\otimes I) U. \]
It is clear that $\Delta$ is a unital positive linear map, thus is a good co-multiplication. By computation, we have that the reduced good convolution $\ast$ is given by
\[M_n(\mathbb{C})\otimes M_n(\mathbb{C})\rightarrow M_n(\mathbb{C}):\quad x\otimes y\mapsto x\ast y={\rm Tr}_2(U x\otimes y U^*), \]
where ${\rm Tr}_2$ is the partial trace at the second position.
In particular, let \[U_\theta=\sqrt{\theta}I+i\sqrt{\theta(1-\theta)}S,\quad 0\leq \theta\leq 1,\]
where $S(A\otimes B)=B\otimes A$ is the swap operator. Let $\Delta_\theta$ be the corresponding co-multiplication. Then the reduced good convolution $\ast_\theta$ is given by
\[x\ast_\theta y={\rm Tr}_2(U_\theta x\otimes y U_\theta^*)=\theta x+(1-\theta)y+i\sqrt{\theta(1-\theta)}[x,y]\]
for any  density matrices $x,y$.
The quantum entropic convolution inequality for this convolution was proved in \cite{ADO16}:
\begin{align*}
H(x\ast_\theta y)\geq \theta H(x)+(1-\theta)H(y).
\end{align*}
However, quantum Young's inequality does not hold for this good convolution.
\end{example}

\subsection{Quantum smooth entropic convolution  inequality}\label{subsec:convolution smooth entropy}
In this section, we will prove the H\"older continuity of von Neumann entropy and convolution entropy. 
Using the continuity of entropy, we establish the quantum smooth entropic convolution inequality.
For a finite dimensional FN $k$-algebra $(\mathcal{M},\tau,\ast,\rho)$,
we call $\tau(I)$ the \textbf{Frobenius-Perron dimension} of $\mathcal{M}$, denoted by $d$. Recall that $\lambda=\min\{\tau(e):\ \text{$e$ is a 
projection in $\mathcal{M}$}\}$. We have
\begin{align}\label{eq:infty norm bound by p norm}
\|x\|\leq \lambda^{-1/p}\|x\|_p,\quad \forall x\in\mathcal{M}.
\end{align}
\begin{theorem}[Quantum smooth entropic convolution inequality]\label{thm: convolution smooth entropy inequality}
Let $(\mathcal{M},\tau,\ast,\rho)$ be a finite dimensional  FN $k$-algebra. 
Let $p,q\in[1,\infty]$, 
 $\epsilon,\eta,\in[0,1]$, $\epsilon+\eta\leq\displaystyle\frac{1}{(d+1)(1+k(d+1))}$.
For any positive operators $x,y\in\mathcal{M}$ with $\|x\|_1=\|y\|_1=k^{-1}$, $0\leq \theta\leq1$, we have 
 \begin{align}\label{eq:quantum smooth convolution entropy inequality}
H_{\epsilon,\eta}^{p,q}(x\ast y)
 \geq&\theta H_{\epsilon}^{p}(x)+(1-\theta)H_{\eta}^{q}(y)-O_{d,\lambda,k}(|\epsilon\log\epsilon|)-O_{d,\lambda,k}(|\eta\log\eta|).
 \end{align}
 \end{theorem}
Before proving  Theorem \ref{thm: convolution smooth entropy inequality}, we  need some preparations.
\subsubsection{Smooth convolution entropy} The convolution entropy of two positive operator $x,y\in\mathcal{M}$ is defined as follows.
\begin{definition}\label{def:convolution entropy}
Let $(\mathcal{M},\tau)$ be a tracial   von Neumann algebra with a good $k$-convolution $\ast$. 
For any two positive operators $x,y\in\mathcal{M}$, 
the \textbf{convolution entropy} is defined by
\begin{align}\label{eq:convolution entropy}
H(x\ast y)=\tau(-x\ast y\log x\ast y). 
\end{align}
\end{definition}

Let  $(\mathscr{P}_{2,\pm},{\rm Tr},\ast,\rho)$ be the  FN $\delta^{-1}$-algebra from subfactor planar algebras. We have that the convolution entropy $H(e_1\ast x)=H(x\ast e_1)=H(x)$, the von Neumann entropy of $x$, where $e_1$ is the Jones projection.

Let $(\mathcal{M},\tau,\ast,\rho)$ be a FN $k$-algebra. From the quantum entropic convolution inequality  (See Theorem \ref{thm:entropy convolution inequality}), we know that the convolution entropy of two positive operators with traces $k^{-1}$ is larger than or equal to the entropy of either of them. 

To study the continuity of convolution entropy, we shall introduce the notion of smooth  convolution  entropy.

\begin{definition}\label{def:convolution_smooth_entropy}
Let $(\mathcal{M},\tau)$ be a tracial  
von Neumann  algebra with a good $k$-convolution $\ast$.
For any positive operators $x,y\in\mathcal{M}$, $\epsilon,\eta\in [0,1]$ and $p,q\in[1,\infty]$,
the \textbf{smooth convolution  entropy}  is defined by
\begin{align}\label{eq:definition smooth convolution entropy}
H_{\epsilon,\eta}^{p,q}(x\ast y):=\inf\{H(z\ast w):\  z,w\in\mathcal{M},\,z,w\geq0,\ \|x-z\|_p\leq\epsilon,\ \|y-w\|_q\leq\eta\}. 
\end{align}
\end{definition}
The smooth entropy is defined as follows.
\begin{definition}\label{def:sup smooth_von_entropy}
Let $(\mathcal{M},\tau)$ be a tracial von Neumann algebra. For any positive operator $x\in\mathcal{M}$, $p\in[1,\infty]$, $\epsilon\in[0,1]$, the \textbf{$(p,\epsilon)$ smooth entropy} of $x$ is defined by 
\begin{align*}
H_{\epsilon}^p(x):=\sup\{H(y):\ y\in\mathcal{M},\ y\geq0,\ \|y-x\|_p\leq \epsilon \}.
\end{align*}
\end{definition}
\begin{remark}
We refer to another smooth entropy studied  in quantum smooth uncertainty principles in \cite{HLW21} and
smooth Renyi entropy studied by R. Renner and S. Wolf in quantum information in \cite{RS04}.
\end{remark}
\subsubsection{Continuity of convolution entropy}\label{subsec:Continuity of convolution entropy}
The following lemma would be useful to prove the H\"older continuity of von Neumann entropy.
 \begin{lemma}\label{lem:tlogt}
 For any $0\leq s\leq t\leq r$, $t-s\leq r/2$,  we have 
\[|t\log t-s\log s|\leq -(t-s)\log (t-s)+2|\log r|(t-s).\]
\begin{proof}
We first assume that $r=1$. Let $\gamma=t-s$, $f(x)=(x+\gamma)\log(x+\gamma)-x\log(x).$ Since $f'(x)=\log(x+\gamma)-\log(x)\geq0$, we have
\begin{align*}
|t\log t-s\log s|=|f(s)|\leq\max\{|f(0)|,|f(1-\gamma)|\}=-(t-s)\log(t-s).
\end{align*}
Applying the above process to $s/r$ and $t/r$, we could obtain the conclusion.
\end{proof}
\end{lemma}
The von Neumann entropy is H\"older continuous.
\begin{proposition}\label{prop:continuity_von_entropy}
Suppose $(\mathcal{M},\tau)$ is a finite dimensional  von Neumann algebra. 
Let $p\in[1,\infty]$, $\epsilon\in[0,1]$, $h>0$.
Let $x,y\in\mathcal{M}$, $x,y\geq0$, $\|x\|_1\leq h$, $\|y\|_1\leq h$. If $\|x-y\|_p\leq \epsilon$, then
\begin{align}\label{eq:continuity of entropy}
|H(x)-H(y)|\leq d^{1-1/p}|\epsilon\log\epsilon|+O_{d,\lambda,h}(\epsilon).
\end{align} 
\begin{proof}
Since every finite dimensional von Neumann algebra is a direct sum of matrix algebras, we may assume that
 \[\mathcal{M}\quad=\quad\sum_{i=1}^m  \oplus\mathop{M_{n_i}(\mathbb{C})}\limits_{\delta_i},\quad n_i\in\mathbb{N}^{*},\quad\delta_i>0.\]
Let ${\rm Tr}_i$ be the unnormalized trace on $M_{n_i}(\mathbb{C})$, we have 
\[\tau=\sum_{i=1}^m\delta_i {\rm Tr}_i,\quad d=\sum_{i=1}^m\delta_i n_i.\]
Let
\[ x=\sum_{i=1}^m x_i,\quad y=\sum_{i=1}^m y_i,\quad x_i,y_i\in M_{n_i}(\mathbb{C}).\]
 Let $\lambda_j(x_i)$ and $\lambda_j(y_i)$ be the $j$-th largest eigenvalues of $x_i$ and $y_i$ respectively. Take 
 \[r=2\lambda^{-1}h\geq2\lambda^{-1}\max\{\|x\|_1,\|y\|_1\}\geq2\max\{\|x\|,\|y\|\}.\] 
The second inequality is due to inequality \eqref{eq:infty norm bound by p norm}. 
  Thus
 \[|\lambda_j(x_i)-\lambda_j(y_i)|\leq r/2.\]
 By Lemma 3.26 in \cite{HLW21}, we have 
  \begin{align}\label{eq:majority 1}
    \|\lambda(x)-\lambda(y)\|_1=\sum_{i=1}^m\delta_i\sum_{j=1}^{n_i}|\lambda_j(x_i)-\lambda_j(y_i)|\leq \|x-y\|_1.
  \end{align}
 Let $f(t)=-t\log t$.  We have 
 \begin{align*}
 |H(x)-H(y)|&=\bigg|\sum_{i=1}^m\delta_i {\rm Tr}_i(x_i\log x_i-y_i\log y_i) \bigg|\\
 &\leq \sum_{i=1}^m\delta_i \sum_{j=1}^{n_i}
\bigg|f(\lambda_j(x_i))-f(\lambda_j(y_i))\bigg| 
\\
& \leq\sum_{i=1}^m\delta_i \sum_{j=1}^{n_i}\bigg(
f\big(| \lambda_j(x_i)-\lambda_j(y_i) |\big)+2|\log r|| \lambda_j(x_i)-\lambda_j(y_i)| \bigg)\quad \text{Lemma \ref{lem:tlogt}}\\
& \leq f(\|\lambda(x)-\lambda(y)\|_1)+
 (\log d+2|\log r|)\|\lambda(x)-\lambda(y)\|_1\quad\text{Jensen's inequality}.
 \end{align*} 
  By H\"older's inequality, we have
  \begin{align}
   \|\lambda(x)-\lambda(y)\|_1\leq\|x-y\|_1\leq d^{1-1/p}\|x-y\|_p\leq d^{1-1/p}\epsilon.
  \end{align}
Thus,
\begin{align*}
f(\|\lambda(x)-\lambda(y)\|_1)&\leq |d^{1-1/p}\epsilon\log d^{1-1/p}\epsilon|+d^{1-1/p}\epsilon\\
&\leq d^{1-1/p}|\epsilon\log\epsilon|+d^{1-1/p}\epsilon(1+(1-1/p)|\log d|).
\end{align*} 
Therefore
\begin{align*}
|H(x)-H(y)|\leq d^{1-1/p}|\epsilon\log\epsilon|+d^{1-1/p}\epsilon(1+(1-1/p)|\log d|+|\log d|+2|\log r|).
\end{align*}
Note that $d^{1-1/p}\leq d+1$, we could obtain the inequality \eqref{eq:continuity of entropy}.
\end{proof}
\end{proposition}

From Proposition \ref{prop:continuity_von_entropy} and primary Young's inequality for good $k$-convolutions, we obtain the H\"older continuity of convolution entropy.
\begin{proposition}\label{prop:continuity convolution entropy}
Suppose $(\mathcal{M},\tau)$ is a finite dimensional von Neumann algebra with a good $k$-convolution $\ast$.
Let $p,q\in[1,\infty]$, $\epsilon, \eta\in[0,1]$, $h>0$,
$\epsilon+\eta\leq\displaystyle\frac{1}{kh(d+1)}$.  
Let $x,y,z,w\in\mathcal{M}$, $x,y,z,w\geq0$, $\|x\|_1,\|y\|_1,\|z\|_1,\|w\|_1\leq h$.
If $\|x-z\|_p\leq\epsilon$, $\|y-w\|_q\leq\eta$, then
\begin{align}\label{eq:continuity of convolution entropy}
|H(x\ast y)-H(z\ast w)|\leq
O_{d,\lambda,h,k}(|\epsilon\log\epsilon|)+O_{d,\lambda,h,k}(|\eta\log\eta|).
\end{align}
\begin{proof}
By primary Young's inequality and H\"older's inequality, one could obtain
\begin{align*}
\|x*y-z*w\|_1&\leq  \|(x-z)\ast y\|_1+\|z\ast(y-w)\|_1\\&\leq kh\|x-z\|_1 + kh\|y-w\|_1\\
&\leq kh(d^{1-1/p}\epsilon+d^{1-1/q}\eta)\\
&\leq kd(d+1)(\epsilon+\eta)\\
&\leq1.
\end{align*}
Note that $\|x\ast y\|_1,\|z\ast w\|_1\leq kh^2$.
Applying Proposition \ref{prop:continuity_von_entropy} to $x\ast y$ and $z\ast w$, we have
\begin{align}\label{eq:continuity of convolution entropy 2}
|H(x\ast y)-H(z\ast w)|&\leq kd(d+1)(\epsilon+\eta)|\log kd(d+1)(\epsilon+\eta)|+O_{d,\lambda,h,k}(\epsilon+\eta).
\end{align}
It is a technical computation from inequality \eqref{eq:continuity of convolution entropy 2} to inequality \eqref{eq:continuity of convolution entropy}.
\end{proof}
\end{proposition}
Now we are ready to prove the quantum smooth entropic convolution inequality.

 \begin{proof}[Proof of Theorem \ref{thm: convolution smooth entropy inequality}]
 Take $h=k^{-1}+d+1$. Then $\|x-z\|_p\leq\epsilon$ implies
 \[\|z\|_1\leq\|x\|_1+\|x-z\|_1\leq k^{-1}+d^{1-1/p}\|x-z\|_p\leq h.\]
 Similarly, $\|y-w\|_q\leq\eta$ implies $\|w\|_1\leq h$.
From Proposition \ref{prop:continuity_von_entropy},
we have
\begin{align}\label{eq:upper bound of the smooth entropy1}
H_\epsilon^p(x)=\sup\{H(z):\ z\in\mathcal{M},\ \|x-z\|_p\leq\epsilon\}\leq H(x)+d^{1-1/p}|\epsilon\log\epsilon|+O_{d,\lambda,k}(\epsilon),
\end{align}
and
\begin{align}\label{eq:upper bound of the smooth entropy2}
H_\eta^q(y)=\sup\{H(w):\ w\in\mathcal{M},\ \|y-w\|_q\leq\eta\}\leq H(y)+d^{1-1/q}|\eta\log\eta|+O_{d,\lambda,k}(\eta).
\end{align}
Note that $\epsilon+\eta\leq 1/kh(d+1)$.
 From Proposition \ref{prop:continuity convolution entropy}, we have
 \begin{align}\label{eq:lower bound smooth convolution entropy}
H_{\epsilon,\eta}^{p,q}(x\ast y)\geq H(x\ast y)-
O_{d,\lambda,k}(|\epsilon\log\epsilon|)-O_{d,\lambda,k}(|\eta\log\eta|).
\end{align}
Combining inequalities \eqref{eq:upper bound of the smooth entropy1}, \eqref{eq:upper bound of the smooth entropy2}, \eqref{eq:lower bound smooth convolution entropy} with Theorem \ref{thm:entropy convolution inequality}, we obtain the inequality \eqref{eq:quantum smooth convolution entropy inequality}
 \end{proof}

 \newpage
 \bibliographystyle{plain}

\end{document}